\documentclass[twoside,12pt,intlimits]{amsart}

\usepackage[utf8]{inputenc}
\usepackage{amssymb}
\usepackage[leqno]{amsmath}
\usepackage{latexsym}
\usepackage{amsthm}
\usepackage{enumitem}
\usepackage{graphicx}
\usepackage[all,cmtip]{xy}
\usepackage{stmaryrd}
\usepackage{bigints}
\usepackage{tikz}

\theoremstyle{plain}
\newtheorem{theorem}{Theorem}

\newtheorem{lemma}[theorem]{Lemma}

\newtheorem{algorithm}[theorem]{Algorithm}

\newtheorem{remark}[theorem]{Remark}

\numberwithin{equation}{section}

\newcommand{\bisect}[1]{\ensuremath{{\textsf{\rm BISECT}(#1)}}}
\newcommand{\refine}[2]{\ensuremath{{\textsf{\rm REFINE}(#1; #2)}}}

\newcommand{\dx}{\ensuremath{\mathrm{d}x}}

\newcommand{\defas}{:=}

\newcommand{\abs}[1]{\ensuremath{\left|#1\right|}}

\providecommand{\tria}{\ensuremath{\mathcal{T}}}
\providecommand{\trias}{\ensuremath{{\mathcal{T}_\star}}}
\providecommand{\bbT}{\mathbb{T}}

\providecommand{\settmp}[2]{{#1\{{#2}#1\}}}
\providecommand{\set}[1]{\settmp{}{#1}}

\newcommand{\setN}{\mathbb{N}}
\newcommand{\setR}{\mathbb{R}}

\usepackage[hyperref]{paper_diening}

\begin{document}

\title[On the threshold condition for D\"orfler marking]{On the threshold condition for \\D\"orfler marking}
\author[L.~Diening]{Lars Diening}
\address{Universit\"at Bielefeld \\ Fakult{\"a}t f{\"u}r Mathematik\\
  D-33615 Bielefeld \\ Germany}%
\email{lars.diening@uni-bielefeld.de}
\author[C.~Kreuzer]{Christian Kreuzer}
\address{TU Dortmund \\ Fakult{\"a}t f{\"u}r Mathematik \\ D-44221 Dortmund \\ Germany}
\email{christian.kreuzer@tu-dortmund.de}



\begin{abstract}
  It is an open question if the threshold condition
  $\theta < \theta_\star$ for the D\"orfler marking parameter is
  necessary to obtain optimal algebraic rates of adaptive finite element methods. We present a
  (non-PDE) example fitting into the common abstract convergence
  framework (axioms of adaptivity) 
  and which is potentially converging with exponential rates. 
  However, for D\"orfler marking $\theta > \theta_\star$ the
  algebraic converges rate can be made arbitrarily small.
\end{abstract}

\thanks{The research of Lars Diening was partially supported by the
  DFG through the CRC 1283}

\maketitle
 
 
 \section{Introduction}
 \label{sec:introduction}

In the pioneering work \cite{Stevenson:07} of Stevenson proved rate
optimality of the standard adaptive finite element method for the
Poisson problem. In contrast to the prior result~\cite{BiDaDeV:04},
which used optimal coarsening based on fast tree
approximation~\cite{BiDeV:04}, Stevenson proved optimality of the
refinement based on D\"orfler marking totally avoiding
coarsening as is standard for stationary problems. One main 
ingredient of the proof is a local estimate, which bounds the distance
of discrete solutions on nested meshes relative to the refined
elements only. As a consequence each refinement that ensures some
error reduction must satisfy a D\"orfler marking condition.  This is
the key in the optimality proof of \cite{Stevenson:07}, since it
allows to compare the adaptive refinement strategy with optimal
refinements when the adaptive marking parameter is below a certain
threshold depending on the ratio of the efficiency and reliability
constants of the a posteriori estimator.

During the last decade, the
approach of Stevenson became very popular and was further developed
e.g. in~\cite{CaKrNoSi:08, DieningKreuzer:2008,%
  BelenkiDieningKreuzer:2012,%
  CasconNochetto:11,KreuzerSiebert:11,%
  FeiFuehPraet:2014,FeischlKarkulikMelenkPraetorius:13,%
  CarstensenPeterseimRabus:2013,BonitoNochetto:2010,Tsogtgerel:2013}.
For and overview of the topic see e.g. the monographs
\cite{NochettoVeeser:12,NoSiVe:09} and for a more exhaustive list of
related works, we refer the reader to \cite{CarFeiPagPra14}.

It is, however, still open whether or not the threshold for the
marking is sharp in general or if it is only a technical artefact.
In~\cite{CarFeiPagPra14}, Carstensen, Feischl, Page, and Praetorius
presented an unifying axiomatic approach for proving rate optimality
of adaptive finite element methods. Therein, the threshold is obtained
with a slightly different technique (compare with
Section~\ref{sec:threshold}) and it is stated that
\smallskip
\begin{center}
  \begin{minipage}{.9\linewidth}
   \textit{\ldots the upper bound for adaptivity parameters which guarantee
   quasi-optimal convergence rates, is independent of the efficiency
   constant.  Such an observation might be a first step to the
   mathematical understanding of the empirical observation that each
   adaptivity parameter \(0<\theta\le0.5\)
   yields optimal convergence rates in the asymptotic regime.}
  \end{minipage}
\end{center}
\smallskip
In this work, we present an example, which satisfies the abstract
axioms of adaptivity from~\cite{CarFeiPagPra14,CarRab17} but also
efficiency, whichs allows to apply the original techniques from~\cite{Stevenson:07}.
However, for any \(\theta\in(0,1)\), we can adjust the parameters such
that, although exponential convergence is possible, the adaptive loop with
D\"orfler marking fails to have arbitrary bad algebraic rates.  It turns out
that, the threshold parameter of~\cite{CarFeiPagPra14,CarRab17} not
involving efficiency suggests for our example values, which are
slightly too conservative. The threshold parameter
of~\cite{Stevenson:07} however, which is the ratio of the reliability
and efficiency constant, is sharp.

We emphasise that the example is not within the context of finite
element discretisations of partial differential equations (compare
also with Remark~\ref{rem:context}). Therefore, technically speaking,
we cannot claim any conclusions for the relevant practical cases,
however, we clarify that the threshold condition can neither be
avoided nor significantly improved within the axiomatic framework
of~\cite{CarFeiPagPra14,CarRab17} even when relying on
efficiency as in~\cite{Stevenson:07}.  Notice that our example also confirms the result of
\cite{DieningKreuzerStevenson:2016,KreuzerSchedensack:2016} that the
maximum marking strategy is more robust in the sense that it provides
optimal convergence rates without restriction on the maximum marking
parameter; compare with Section~\ref{sec:maximums-strategy}.

\section{Axioms of adaptivity}
\label{sec:axiomatic-approach}
In this section, we recall the axiomatic approach in~\cite{CarFeiPagPra14,CarRab17}
of the proof of optimal convergence rates for adaptive finite element
methods. The presentation mainly follows 
\cite{CarRab17} neglecting the additional refinement indicators but is
simplified tailored to our needs, i.e., we do not cover the full
generality of~\cite{CarFeiPagPra14,CarRab17}. In particular, we
consider the most simple case of
dimension \(d=1\)  and let \(\Omega\subset \setR\) be a
non-empty open interval. Before we state the axioms, we first verify
the refinement conditions.

\subsection{Refinement by bisection}
\label{sec:refinement}

Let \(\mathcal{T}_0\) be an initial partition of \(\Omega\) into closed
intervals (called macro elements) 
and denote by \(\bbT\) the set of its possible refinements.
To be more precise, we
introduce bisection of a closed interval \([a,b]\),
\(a<b\) by
\begin{align*}
  \bisect{[a,b]}=\{[a,\tfrac{a+b}2],\,[\tfrac{a+b}2,b]\}.
\end{align*}
We say that \(\mathcal{T}_\star\) is a refinement of \(\mathcal{T}\)
(or \(\mathcal{T}_\star\ge \mathcal{T}\))
iff there exist a finite sequence
of partitions \(\{\mathcal{T}_n\}_{n=1}^N\) and \(T_n\in
\mathcal{T}_n\), \(n=1,\ldots,N-1\), such that \(
\mathcal{T}_\star=\mathcal{T}_N \) and \(\mathcal{T}_1=\mathcal{T}\)
as well as
\begin{align*}
  \mathcal{T}_{n+1}=(\mathcal{T}_{n}\setminus \set{T_{n}})\cup \bisect{T_{n}},
  \quad n=1,\ldots, N-1.
\end{align*}
With this definition \((\bbT,\le)\) becomes a lattice and we can
define for \(\mathcal{T}_a,\mathcal{T}_b\in \bbT\) 
\begin{align*}
  \mathcal{T}_a\wedge \mathcal{T}_b&\defas
  \operatorname{arg\,max}\{\mathcal{T}'\in\bbT\colon \mathcal{T}'\le
                                     \mathcal{T}_a~\text{and}~\mathcal{T}'\le \mathcal{T}_b\}
                                     \intertext{and}
     \mathcal{T}_a\vee \mathcal{T}_b&\defas
  \operatorname{arg\,min}\{\mathcal{T}'\in\bbT\colon 
                                     \mathcal{T}_a\le \mathcal{T}'~\text{and}~
                                        \mathcal{T}_b\le \mathcal{T}'\};                    
\end{align*}
i.e. the finest common coarsening respective the coarsest common
refinement. Moreover, we have
\begin{align*}
  \#(\mathcal{T}_a\vee
  \mathcal{T}_b)=\#\mathcal{T}_a+\#\mathcal{T}_b-\#(\mathcal{T}_a\wedge
  \mathcal{T}_b)\le \#\mathcal{T}_a+\#\mathcal{T}_b-\#\mathcal{T}_0.
\end{align*}
Thanks to the bisection rule, we can also
recursively assign to each \(T\in\mathcal{T}\), \(\mathcal{T}\in\bbT\), a
generation by
\begin{align*}
  g(T)=0~\text{if}~T\in\mathcal{T}_0 \qquad\text{and}\qquad
  g(T)=g(T')+1~\text{if}~T\in \bisect{T'}.
\end{align*}

Defining for \(\mathcal{T}\in\bbT\) and
\(\mathcal{M}\subset\mathcal{T}\) the refinement procedure
\begin{align*}
  \refine{\mathcal{T}}{\mathcal{M}}\defas\operatorname{arg\,min}\{\mathcal{T}'\in\bbT\colon
  \mathcal{T}\le
  \mathcal{T}'~\text{and}~\mathcal{M}\cap\mathcal{T}'=\emptyset\}, 
\end{align*}
we obviously have
\begin{align}\label{eq:refine}
  \refine{\mathcal{T}}{\mathcal{M}}=(\mathcal{T}\setminus\mathcal{M})\cup\bisect{\mathcal{M}}
\end{align}
with \(\bisect{\mathcal{M}}\defas \bigcup\{\bisect{T}:T\in \mathcal{M}\}\).
Obviously, we thus have
\(\mathcal{T}_\star=\refine{\mathcal{T}}{\mathcal{M}}\in\bbT\) and
\begin{align*}
  \#\mathcal{T}_\star -\#\mathcal{T}=\#\mathcal{M}. 
\end{align*}

We conclude that our refinement framework satisfies the requirements
in \cite[Section 2.4]{CarFeiPagPra14}.

\subsection{Adaptive algorithm with D\"orfler marking}
\label{sec:afem}
In the following we formulate the basic conditions from
\cite{CarFeiPagPra14,CarRab17} sufficient for optimal 
convergence rates of the adaptive D\"orfler marking strategy. 
The precise algorithm and the optimality result is stated in
section~\ref{sec:axioms} below. 

We assume that for any \(\mathcal{T}\in\bbT\), and any element
\(T\in\mathcal{T}\), we have nonnegative indicators
\(\eta_{\mathcal{T}}(T)\) available and set
\begin{align*}
  \eta_{\mathcal{T}}^2(\mathcal{M})=\sum_{T\in
  \mathcal{M}}\eta_{\mathcal{T}}^2(T)\quad\text{for
  any}~\mathcal{M}\subset \mathcal{T}.
\end{align*}
Moreover, we assume that there is a nonnegative distance measure on
\(\bbT\) denoted by \(\delta(\mathcal{T},\mathcal{T}_\star)\) for
\(\mathcal{T},\mathcal{T}_\star\in\bbT\). This distance measures in
the application the error between to discrete solutions.

Based on the above indicators, we can formulate the adaptive
algorithm.
\begin{algorithm}[AFEM with D\"orfler marking] 
  \label{alg:AFEM}
  Let \(\mathcal{T}_0\) be an initial triangulation of \( \Omega\) and $\theta \in(0,1)$ a
  given marking parameter. Set $k:=0$ and iterate
  \begin{itemize}[leftmargin=.5cm]
  \item Compute the indicators
    \(\{\eta_{\mathcal{T}_k}(T):T\in\mathcal{T}_k\}\).
  \item Choose \(\mathcal{M}_k\subset\mathcal{T}_k\) such that 
    \begin{align}
      \label{eq:doerfler}
      \theta\eta_{\mathcal{T}_k}^2(\mathcal{T}_k)\le \eta_{\mathcal{T}_k}^2(\mathcal{M}_k)
    \end{align}
    with quasi-minimal cardinality, i.e.,
    \(\#\mathcal{M}_k\le C_{\texttt{DM}}\,\#\mathcal{M}\) for some fixed constant
    \(C_{\texttt{DM}} \ge1\) and all \(\mathcal{M}\subset\mathcal{T}_k\) with
    \(\theta\eta_{\mathcal{T}_k}(\mathcal{T}_k)\le
    \eta_{\mathcal{T}_k}(\mathcal{M})\).

    If the set~$\mathcal{M}_k$ has minimal cardinality (i.e. \(C_{\texttt{DM}}=1\))
    then~\eqref{eq:doerfler} is called \emph{optimal
      D\"orfler marking}.
  \item Construct the refinement
    \begin{align*}
      \mathcal{T}_{k+1} =\refine{\mathcal{T}_k}{\mathcal{M}_k}
    \end{align*}
    and set \(k:=k+1\).
  \end{itemize}
\end{algorithm}

\subsection{The axioms}
\label{sec:axioms}

In this section, we present the axioms of adaptivity from
\cite{CarFeiPagPra14,CarRab17} in a simplified version tailored to our
needs.  We assume that the indicators and the distance measure from
the previous section~\ref{sec:afem} satisfy the following conditions:
\begin{enumerate}[label={(A\arabic*)}]
\item\label{itm:A1}  \textbf{Stability.} For all \(\mathcal{T},\mathcal{T}_\star\in\bbT\) with
\(\mathcal{T}_\star\ge \mathcal{T}\), we have
  \begin{align*}
    \abs{\eta_{\mathcal{T}}(\mathcal{T}\cap\mathcal{T}_\star)-\eta_{\mathcal{T}_\star}(\mathcal{T}\cap\mathcal{T}_\star)}\le \delta(\mathcal{T},\mathcal{T}_\star)
  \end{align*}
\item\label{itm:A2} \textbf{Reduction.} There exists \(\rho\in[0,1)\) such that for all \(\mathcal{T},\mathcal{T}_\star\in\bbT\) with
\(\mathcal{T}_\star\ge \mathcal{T}\), we have 
  \begin{align*}
    \eta_{\mathcal{T}_\star}(\mathcal{T}_\star\setminus\mathcal{T})\le
    \rho \eta_{\mathcal{T}}(\mathcal{T}_\star\setminus\mathcal{T}).
  \end{align*}
\item \label{itm:A3} \textbf{Discrete reliability.} There exists \(C_3>0\), such
  that for all \(\mathcal{T},\mathcal{T}_\star\in\bbT\) with
  \(\mathcal{T}_\star\ge \mathcal{T}\), we have
  \begin{align*}
    \delta(\mathcal{T},\mathcal{T}_\star)^2\le C_3 \eta_{\mathcal{T}}^2(\mathcal{T}\setminus\mathcal{T}_\star).
  \end{align*}
\item\label{itm:A4} \textbf{Quasi-orthogonality.} There exists \(C_4>0\), such
  that for any sequence
  \(\{\mathcal{T}_k\}_k\subset\bbT\) of nested partitions
  (i.e. \(\mathcal{T}_1\le\mathcal{T}_2\le\ldots\)), we have for all
  \(\ell\in\setN\) that
  \begin{align*}
    \sum_{k=\ell}^\infty \delta(\mathcal{T}_{k+1},\mathcal{T}_k)^2\le
    C_4\eta_{\mathcal{T}_\ell}^2(\mathcal{T}_\ell). 
  \end{align*}
\end{enumerate}

\begin{remark}
  We note that  \ref{itm:A1}--\ref{itm:A4} correspond to the
  respective conditions in \cite{CarRab17} with
  \begin{gather*}
    \tag{A1}\Lambda_1=1\\
    \tag{A2}\rho_2=\rho,\qquad\Lambda_2=0
    \\
    \tag{A3} \Lambda_{ref}=1,\quad \Lambda_3=C_3,\quad\text{and}\quad
    \hat\Lambda_3=0
    \\
    \tag{A4}
    \Lambda_4=C_4.
  \end{gather*}
  The conditions (B1) and (B2) in \cite{CarRab17} do not apply since
  we assume \(\mu_\ell\equiv 0\) for the second indicator in \cite{CarRab17}.
  Note that therefore also the quasi-monotonicity
  (QM) condition in \cite{CarRab17} is satisfied automatically since
  we may chose \(\hat\Lambda_3=0\) in \cite[Theorem 3.2]{CarRab17}.
\end{remark}

We recall the following main theorem from
\cite{CarFeiPagPra14,CarRab17}. 

\begin{theorem}\label{thm:optimality} Suppose that
\ref{itm:A1}--\ref{itm:A4} hold and define the threshold
\begin{align*}
  \theta_\star\defas\frac{1}{1+C_3}.
\end{align*}
Then Algorithm~\ref{alg:AFEM} is rate optimal if
\(\theta<\theta_\star\), i.e., in this case we have for all
\(s>0\) there exists \(C>0\) with
\begin{align*}
  \lefteqn{\sup_{k\in\setN} \big(
  (\#\mathcal{T}_k-\#\mathcal{T}_0)^s\eta_{\mathcal{T}_k}(\mathcal{T}_k)\big)} 
  \qquad
  \\
  &\leq C
    \sup_{N\in\setN} \big(
    N^s\min\{\eta_{\mathcal{T}}(\mathcal{T})\colon
    \#\mathcal{T}-\#\mathcal{T}_0\le N\} \big).
\end{align*}
\end{theorem}

The original approach of Stevenson is slightly different in that it
utilizes also the efficiency of the estimator. This allows to show
convergence rates for the error rather than for the estimator; compare
with Remark~\ref{rem:delta=eta}. To
work in the framework of Stevenson we modify/sharpen two of the axioms.

We replace the stability~\ref{itm:A1} by the following efficiency condition.
\begin{enumerate}[label={(A\arabic*')},start=1]
\item \label{itm:A1'} \textbf{Efficiency.} For all \(\mathcal{T}\in
  \bbT\), we have 
  \begin{align*}
    C_1\eta_{\mathcal{T}}^2(\mathcal{T})\le \delta(\mathcal{T})^2
  \end{align*}
  for some constant \(C_1>0\). 
\end{enumerate}
Typically, $\delta(\mathcal{T})$ measures the error, e.g. in
applications the distance between the discrete solution and the exact
solution of the PDE. In our example, $\delta(\mathcal{T})$ is the
distance to the `finest' partition
\begin{align}\label{df:dT}
  \delta(\tria)\defas
  \inf\{\delta(\mathcal{T}_\star,\mathcal{T}):\mathcal{T}_\star\in\bbT~\text{with}~\mathcal{T}
  \ge
  \mathcal{T}\}.  
\end{align}
As a consequence of the lattice structure of \((\bbT,\ge)\), we have
that the definition~\eqref{df:dT} is unique when
\(\delta(\cdot,\mathcal{T}):\bbT\to \setR\) is non-increasing under
refinement. This property is immediate when replacing
the quasi-orthogonality condition~\ref{itm:A4} by
the following orthogonality property; compare also with Remark~\ref{rem:delta=eta} below.
\begin{enumerate}[label={(A\arabic*')},start=4]
\item \label{itm:A4'}  \textbf{Orthogonality.} For all
  \(\mathcal{T},\mathcal{T}_\star,\mathcal{T}_{\circ}\in\bbT\) with
  \(\mathcal{T}\le \mathcal{T}_\star\le \mathcal{T}_{\circ}\), we have
  \begin{align*}
    \delta(\mathcal{T}_{\circ},\mathcal{T}_\star)^2+\delta(\mathcal{T}_\star,\mathcal{T})^2
    =\delta(\mathcal{T}_\circ,\mathcal{T})^2.
  \end{align*}
\end{enumerate}

Then Stevenson proved in \cite{Stevenson:07} the following version of
Theorem~\ref{thm:optimality} with a different threshold.
\begin{theorem}\label{thm:optRob}
  Suppose that
  \ref{itm:A2}--\ref{itm:A3} and assume in addition \ref{itm:A1'}
  and \ref{itm:A4'}. Define the threshold
  \begin{align*}
     \tilde\theta_\star\defas\frac{C_1}{C_3}.
  \end{align*}
  Then Algorithm~\ref{alg:AFEM} is rate optimal if
  \(\theta<\tilde\theta_\star\), i.e., in this case we have for all
  \(s>0\) there exists \(C>0\) with
  \begin{align*}
    \lefteqn{\sup_{k\in\setN}
    (\#\mathcal{T}_k-\#\mathcal{T}_0)^s \delta(\mathcal{T}_k)
    }
    \qquad
    \\
    &\leq C
      \sup_{N\in\setN}
      N^s\min\{ \delta(\mathcal{T})
      \colon
      \#\mathcal{T}-\#\mathcal{T}_0\le N\}.
  \end{align*}
\end{theorem}

\begin{remark}[\(\delta(\mathcal{T})\)
  vs. \(\eta_{\mathcal{T}}(\mathcal{T})\)]
  \label{rem:delta=eta}
  We remark that in \cite{Stevenson:07} optimal convergence rates are
  proved for $\delta(\mathcal{T})$ in contrast to
  \cite{CarFeiPagPra14,CarRab17}, which focus
  on~$\eta_{\mathcal{T}}(\mathcal{T})$. Let us compare these two approaches.

  Since \(\delta(\mathcal{T}_\star,\mathcal{T})\ge 0\) for
  \(\mathcal{T}_\star,\mathcal{T}\in\bbT\) with
  \(\mathcal{T}_\star\ge \mathcal{T}\), we conclude from~\ref{itm:A4'}
  that \(\delta :\bbT\to\setR_\ge\) is monotone decreasing under
  refinement.  Moreover, recalling~\eqref{df:dT}, it follows
  from~\ref{itm:A3} that
  \begin{align*}
    \delta(\mathcal{T})\le C_3\eta_{\mathcal{T}}(\mathcal{T}).
  \end{align*}
  This is an upper bound or equivalently~\ref{itm:A4}
  with \(C_4=C_3\).

  Combining this with the efficiency~\ref{itm:A1'}, we have
  equivalence of the error and the estimator, i.e.
  $C_1 \eta^2_{\mathcal{T}}(\mathcal{T}) \leq \delta(\mathcal{T})^2 \leq
  C_3 \eta^2_{\mathcal{T}}(\mathcal{T})$.
  
  As a consequence,
  we have that $\mathcal{\delta}(\mathcal{T})$ converges iff
  $\eta_{\mathcal{T}}(\mathcal{T})$ converges and both converge then
  with the same rates. Taking the setting \ref{itm:A1}--\ref{itm:A4}
  of \cite{CarFeiPagPra14,CarRab17}, however, the convergence behavior
  of \(\delta\) and \(\eta\)
  may differ. In particular, in view of~\ref{itm:A3}, the convergence
  rate of \(\delta\) may be better than the one of \(\eta\). 
\end{remark}

\subsection{The origin of the thresholds}
\label{sec:threshold}

We start with discussing the threshold \(\theta_\star\) from Theorem~\ref{thm:optimality}.
Let \(\mathcal{T}_k\in\bbT\) be from Algorithm~\ref{alg:AFEM} for some
\(k\in\setN\) and assume \(\mathcal{T}_\star\ge\mathcal{T}_k\) with
\begin{align}\label{eq:estred}
  \eta_{\mathcal{T}_\star}(\mathcal{T}_\star)\le \kappa\eta_{\mathcal{T}_k}(\mathcal{T}_k)
\end{align}
for some arbitrarily fixed \(\kappa\in(0,1)\). Then from
\(\eta_{\mathcal{T}_\star}(\mathcal{T}_\star\cap\mathcal{T}_k)\le
\eta_{\mathcal{T}_\star}(\mathcal{T}_\star)\) and \ref{itm:A1} for
\(0<\gamma<\frac{1}{\kappa^2}-1\) in Youngs inequality we conclude
that
\begin{align*}
  (1-(1+\gamma)\kappa^2)\eta_{\mathcal{T}}^2 (\mathcal{T}_k)
  &\le
    \eta_{\mathcal{T}_k}^2 (\mathcal{T}_k)-(1+\gamma)\eta_{\mathcal{T}_\star}^2 (\mathcal{T}_\star)
  \\
  &\le
    \eta_{\mathcal{T}_k}^2 (\mathcal{T}_k)-(1+\gamma)\eta_{\mathcal{T}_\star}^2 (\mathcal{T}_\star\cap\mathcal{T}_k)
  \\
  &\le \eta_{\mathcal{T}_k}^2 (\mathcal{T}_k\setminus\mathcal{T}_\star)+\frac{1}{1+\gamma^{-1}}\delta(\mathcal{T}_\star,\mathcal{T}_k)^2.
\end{align*}
Now applying~\ref{itm:A3}, we obtain
\begin{align}\label{theta*Doerfler}
  \frac{1-(1+\gamma)\kappa^2}{1+\frac{C_3}{1+\gamma^{-1}}}\,\eta_{\mathcal{T}_k}^2 (\mathcal{T})\le \eta_{\mathcal{T}_k}^2 (\mathcal{T}_k\setminus\mathcal{T}_\star).
\end{align}
In other words, the set of elements
\(\mathcal{T}_k\setminus\mathcal{T}_\star\) from \(\tria_k\) which  
are refined in \(\mathcal{T}_\star\) satisfies a D\"orfler
marking property. When
\begin{align}\label{eq:Doerfler-cond}
  \theta\le \frac{1-(1+\gamma)\kappa^2}{1+\frac{C_3}{1+\gamma^{-1}}},
\end{align}
then the quasi minimal cardinality of \(\mathcal{M}_k\) in Algorithm~\ref{alg:AFEM} implies
\begin{align*}
  \#(\mathcal{T}_k\setminus\mathcal{T}_\star)\le C \#\mathcal{M}_k,
\end{align*}
which is the key in the proof of the rate optimality Theorem~\ref{thm:optimality}; compare e.g. with
\cite{Stevenson:07,CaKrNoSi:08,CarFeiPagPra14}. 
By choosing  \(\kappa>0\) small, we observe
that~\eqref{eq:Doerfler-cond} can only hold if 
\begin{align}\label{eq:threshold-axioms}
  \theta<\theta_\star=\frac1{1+C_3}.
\end{align}

In order to discuss the threshold \(\tilde\theta_\star\) from
Theorem~\ref{thm:optRob}, instead of satisfying the estimator
reduction~\eqref{eq:estred}, we assume that \(\tria_\star\in\bbT\),
\(\tria_\star\ge \tria_k\) reduces the distance
\begin{align*}
  \delta(\mathcal{T}_\star)\le \kappa\delta(\mathcal{T})
\end{align*}
for some \(\kappa\in(0,1)\) arbitrarily fixed. We then
have 
\begin{align*}
  (1-\kappa^2)C_1\,\eta_{\mathcal{T}}^2(\mathcal{T})&\le 
  (1-\kappa^2)
                                                      \delta(\mathcal{T})^2
  \\
  &\le
                                      \delta(\mathcal{T})^2-\delta(\mathcal{T}_\star)^2
    = \delta(\mathcal{T}_\star,\mathcal{T})^2
  \\
  &\le C_3 \eta_{\mathcal{T}}^2(\mathcal{T}\setminus\mathcal{T}_\star).
\end{align*}
Arguing as before, there exists \(\kappa>0\) such that the above
computation implies a D\"orfler condition
\(\theta\eta_{\tria_k}^2(\tria_k)\le
\eta^2_{\tria_k}(\tria_k\setminus\tria_\star)\) only if \(\theta\le \tilde\theta_\star=\frac{C_1}{C_3}\).

\section{D\"orfler marking with suboptimal convergence rates}
\label{sec:example}

For a given marking parameter \(\theta\in(0,1)\), and \(s_0>0\), we
construct an example with an exponential optimal convergence rate that
satisfies the axioms of adaptivity \ref{itm:A1}--\ref{itm:A4} and also
\ref{itm:A1'}+\ref{itm:A4'} with
\(\delta(\mathcal{T})=\eta_{\mathcal{T}}(\mathcal{T})\),
\(\mathcal{T}\in\bbT\) (i.e. \(C_1=1\)) and local reliability constant~$C_3 = K > 0$.

Thus, in this situation \(\frac{1}{K}=\tilde{\theta}_\star\ge
\theta_\star = \frac{1}{1+K}\), i.e. the threshold of Theorem~\ref{thm:optRob}
is less conservative than the one in Theorem~\ref{thm:optimality}. In
particular, thanks to the possible exponential convergence, if
\(\theta\le\tilde\theta_\star\) then
Algorithm~\ref{alg:AFEM} converges with any possible 
algebraic rate~$s>0$.

However, we will see that for any $\theta\in(0,1)$, the
example can be adjusted with arbitrary close \(\tilde{\theta}_\star<\theta\), such that the
adaptive Algorithm~\ref{alg:AFEM} will not converge with rate~$s_0$,
more precisely
\begin{align*}
  \sup_{k\in\setN}
  (\#\mathcal{T}_k-\#\mathcal{T}_0)^{s_0}\eta_{\mathcal{T}_k}(\mathcal{T}_k)
  = 
  \sup_{k\in\setN}
  (\#\mathcal{T}_k-\#\mathcal{T}_0)^{s_0}\delta(\mathcal{T}_k)
  = \infty.
\end{align*}
This shows that a threshold conditions as in Theorems~\ref{thm:optimality}
and~\ref{thm:optRob} cannot be avoided in the axiomatic framework of
\cite{CarFeiPagPra14,CarRab17} and, moreover, can be arbitrarily restrictive. 

\subsection{The setup}
\label{sec:setup}

For \(\Omega=(0,M+1)\), \(M\in\setN\), consider the initial
partition
\begin{align}\label{df:T0}
  \mathcal{T}_0=\{[0,1],[1,2],\ldots, [M,M+1]\}
\end{align}
and denote the set of admissible refinements according to
Section~\ref{sec:refinement} by \(\bbT\).
For \(A\subset \Omega\) and \(\mathcal{T}\in\bbT\),
we use the notation
\begin{align*}
  \mathcal{T}|_A &:= \set{T \in \mathcal{T}\,:\, T \subset \overline{A}}.
\end{align*}
We denote by $T_0(\tria)$ the element of~$\tria$ that
contains zero and by $g_0(\tria) := g(T_0(\mathcal{T}))$ its generation.

For fixed $\alpha, \beta>0$ and \(K>1\), we define
\begin{align}\label{df:example}
  \eta_\tria^2(T)
  &:=
    \begin{cases}
      2^{-\alpha g_0(\tria)- \beta(g(T)+\frac{m-1}M)} \abs{T} &\text{if
        $T \subset [m,m+1]$, \(m\ge1\)}
      \\
      \frac1{K-1}\, \eta_\tria^2(\tria|_{[1,M+1]})
      &\text{if $T=T_0(\mathcal{T})$},
      \\
      0&\text{else}.
    \end{cases}
\end{align}
The constant~$K$ will be the reliability constant, i.e. $C_3=K$.  The
constants~$\alpha>0$ and $M \in \setN$, will be chosen later depending
on the convergence rate~$s_0$ and the marking parameter~$\theta$.

This yields the immediate relation
\begin{align}
  \label{eq:T_0=1/M}
  \eta_{\mathcal{T}}^2(T_0(\mathcal{T}))=\frac{1}{K}\,
  \eta_{\mathcal{T}}^2(\mathcal{T}). 
\end{align}
In particular, the estimator of the element~$T_0(\mathcal{T})$ is
comparable to the estimator on all of~$\mathcal{T}$.

For $\tria_* \geq \tria$ we define
\begin{align}\label{df:delta}
  \delta(\trias,\tria)^2 &\defas
                           \eta_{\tria}^2 (\tria)-\eta_{\trias}^2 (\trias)\ge
                           0.
\end{align}
Thus, \(\delta(\tria)^2 \defas \eta_{\tria}^2 (\tria)\).

Note that a refinement of~$T_0(\mathcal{T})$ decreases
all estimators by a factor of~$2^{-\alpha}$. Thus, for suitable
refinement, the estimator \(\eta\) defined in~\eqref{df:example}
converges exponentially.
\begin{lemma}[Exponential Convergence]\label{l:exp-conv}
  Assume that \(\{\mathcal{T}_k\}_{k\in\setN_0}\) is generated by a
  repeatedly refinement of~$T_0(\mathcal{T}_k)$, i.e
  \begin{align*}
    \mathcal{T}_i=\refine{\{T_0(\mathcal{T}_{i-1})\}}{\mathcal{T}_{i-1}}. 
  \end{align*}
  Then the estimator and the distance converge exponentially, i.e.
  \begin{align*}
    \delta(\mathcal{T}_k)^2 = \eta_{\mathcal{T}_k}^2 (\mathcal{T}_k)= 2^{-\alpha k
    }\eta_{\mathcal{T}_0}^2 (\mathcal{T}_0)\quad\text{and}\quad
    \#\mathcal{T}_k-\#\mathcal{T}_0=k. 
  \end{align*}
  In particular, we have for all rates \(s>0\) that
  \begin{align*}
    \sup_{N\in\setN}
    N^s\min\{\eta_{\mathcal{T}}(\mathcal{T})\colon
    \#\mathcal{T}-\#\mathcal{T}_0\le N\}<\infty.
  \end{align*}
\end{lemma}
\begin{proof}
  Observing that 
  \begin{align*}
    \#\mathcal{T}_k-\#\mathcal{T}_0=k=g(T_0(\mathcal{T}_k)),
  \end{align*}
  the assertion 
  is an immediate consequence of~\eqref{df:example} and~\eqref{eq:refine}.
\end{proof}
\begin{remark}
  We are using in our setup the bisection method without conforming
  closure. This is just for the sake of a clear presentation. All
  observations remain valid if a conforming closure step is included.
\end{remark}

\subsection{Verifying the axioms}
\label{sec:better-axioms}

In order to verify the conditions~\ref{itm:A1}--\ref{itm:A4} as
well as~\ref{itm:A1'} and~\ref{itm:A4'},
we first observe the estimator defined in~\eqref{df:example} is
locally non-increasing under refinement.


\begin{lemma}[Monotonicity]
  \label{lem:loc-monotone}
 Let $\tria_\star\ge
  \tria$ such  that  $T\in\tria$ is bisected
  into~$\{T_1,T_2\}=\bisect{T}\subset\tria_\star$, then
  \begin{align*}
    \eta^2_{\trias}(T_1) +
    \eta^2_{\trias}(T_2) &\leq \eta_\tria^2(T).
  \end{align*}
  In particular, for all \(\mathcal{T}_\star\ge \mathcal{T}\), we have
  \(\eta_{\trias}(\trias)\leq \eta_{\tria}(\tria)\).
\end{lemma}

\begin{proof}
  We consider first the case $0\not\in T$, then 
  \begin{align}\label{eq:0inrefd}
    \eta^2_\trias(T_i) 
    &\leq 2^{-1-\beta}
      \eta^2_\tria(T), \qquad i=1,2.
      \end{align}
      Therefore, we conclude from \(\alpha\ge0\) that
      \begin{align*}
      \eta^2_\trias(T_1)+\eta^2_\trias(T_2)&\le   2\,2^{-1-\beta}
      \eta^2_\tria(T)  = 2^{-\beta}
      \eta^2_\tria(T)\le \eta^2_\tria(T).
  \end{align*}

  Assume now that $0\in T$ (i.e. $T=T_0(\mathcal{T})$). We first
  observe  for unrefined elements
  \begin{align}\label{eq:unrefd}
    \tilde
    T\in\tria\cap\trias\qquad\Rightarrow\qquad\eta^2_\trias(\tilde T) 
    &\leq
      \eta^2_\tria(\tilde T).
  \end{align}
  W.l.o.g. let $0\in
  T_1$, then we have from~\eqref{eq:unrefd} and
  \eqref{eq:0inrefd}, that \(\eta^2_\trias(T_2)=0\) and thus
  \begin{align}\label{refT0}
    \begin{aligned}
      \eta^2_\trias(T_1) +\eta^2_\trias(T_2)
      &=\eta^2_\trias(T_1) =\frac1{K-1}\,
      \eta^2_\trias(\trias|_{[1,M+1]})
      \\
      &\le 2^{-\alpha} \frac1{K-1}\,
      \eta^2_\tria(\tria|_{[1,M+1]})
      \\
      &\le \frac1{K-1}\, \eta^2_\tria(\tria|_{[1,M+1]})=\eta^2_\tria(T_0(\mathcal{T}))
    \end{aligned}
  \end{align}
  This finishes the proof.
\end{proof}

We are now in the position to verify the axioms of adaptivity from Section~\ref{sec:axioms}.
\begin{enumerate}[label={\textbf{(A\arabic{*})}}]
\item \textbf{Stability.}
  We recall  from~\eqref{eq:unrefd} that $\eta^2_\trias(T)
  \geq \eta^2_\tria(T)$ 
  for each unrefined~$T\in \tria\cap\trias$.   Moreover, by the local
  monotonicity (Lemma~\ref{lem:loc-monotone}), we have 
  $\eta^2_{\trias}(\trias \setminus \tria) \leq \eta^2_{\tria}(\tria
  \setminus \trias)$ and therefore
  \begin{align*}
    \lefteqn{\abs{  \eta^2_{\tria}(\tria \cap \trias) -  \eta^2_{\trias}(\tria
    \cap \trias)}} \qquad
    \\
    &=\eta^2_{\tria}(\tria \cap \trias) -  \eta^2_{\trias}(\tria
    \cap \trias)
    \\
    &\leq  \eta^2_{\tria}(\tria \cap \trias) -  \eta^2_{\trias}(\tria
      \cap \trias) +  \eta^2_{\tria}(\tria \setminus \trias) -  \eta^2_{\trias}(\trias
      \setminus \tria)
    \\
    &=  \eta^2_{\tria}(\tria) -  \eta^2_{\trias}(\trias)
    \\
    &= \delta^2(\tria,\trias).
  \end{align*}
  This and $\abs{a-b} \leq \sqrt{\abs{a^2-b^2}}$ for $a\ge b \ge0$ imply~\ref{itm:A1}.
 
\item \textbf{Reduction.}
  Assume first, that $T_0(\tria)$ is not refined in
  $\trias\ge\tria$.  Then we have 
  \begin{align*}
    \eta^2_\trias(\trias \setminus \tria) &\leq 2^{-\beta}
    \eta^2_\tria(\tria \setminus \trias) .
  \end{align*}
  If on the other hand $T_0(\tria)$ is refined in $\trias$, then each
  estimator is at least reduced 
  by the factor~$2^{-\alpha}$, and thus similar to~\eqref{refT0}, we obtain
  \begin{align*}
    \eta^2_\trias(\trias \setminus \tria) &\leq 2^{-\alpha}
    \eta^2_\tria(\tria \setminus \trias) .
  \end{align*}
  Thus in both cases we conclude~\ref{itm:A2} with
  $\rho=2^{-\frac{\min\set{\alpha,\beta}}2}$.
\item \textbf{Discrete reliability.}
  Assume first that \(T_0(\tria)\) is not refined in \(\trias\ge\tria\), i.e. $T_0(\tria)\in\trias$. Then 
  $\eta^2_\tria(\tria \cap
  \trias) = \eta^2_\trias(\tria \cap \trias)$ and 
  \begin{align*}
    \delta^2(\tria,\trias)
    &= \eta^2_\tria(\tria) - 
      \eta^2_\trias(\trias)
    \\
    &= \eta^2_\tria(\tria \setminus \trias) + \eta^2_\tria(\tria \cap
      \trias) - 
      \eta^2_\trias(\trias \setminus \tria ) - \eta^2_\trias(\tria \cap
      \trias)
    \\
    &= \eta^2_\tria(\tria \setminus \trias)  - 
      \eta^2_\trias(\trias \setminus \tria)
    \\
    &\leq \eta^2_\tria(\tria \setminus \trias)
  \end{align*}
  
  If otherwise $T_0(\tria)\in\tria\setminus\trias$, then we obtain
  with~\eqref{df:example} that
  \begin{align*}
    \delta^2(\tria,\trias)
    &\leq \eta^2_\tria(\tria)
    \\
    &= K \,\eta^2_\tria(T_0(\tria))
    \\
    &\leq K\, \eta^2_\tria(\tria \setminus \trias).
  \end{align*}
  In other words, we have~\ref{itm:A3} with \(C_3=\max\{1,K\}=K\).

\item \textbf{Quasi-orthogonality.}
  For a sequence of nested meshes $
  \tria_1\le\tria_2\le\cdots$, in \(\bbT\), we have 
  \begin{align*}
    \sum_{k=1}^N \delta(\mathcal{T}_{k+1}, \tria_k)^2&=\sum_{k=1}^N
                                                       \eta_{\tria_k}^2 (\tria_k)-\eta_{\tria_{k+1}}^2 (\tria_{k+1})
                                                       \\
                                                     &=\eta_{\tria_1}^2 (\tria_1)-\eta_{\tria_{N+1}}^2 (\tria_{N+1})
    \\
    &\le \eta_{\tria_1}^2 (\tria_1).
  \end{align*}
  Taking the limit \(N\to\infty\) and observing from Lemma~\ref{lem:loc-monotone} that
  \(\delta(\mathcal{T}_{k+1}, \tria_k)^2\ge0\), we
  conclude~\ref{itm:A4} with \(C_4=1\).
\end{enumerate}

Also~\ref{itm:A1'} and~\ref{itm:A4'} are satisfied by the error
indicators.
\begin{enumerate}[label={\textbf{(A\arabic*')}},start=1]
\item \textbf{Efficiency.} For \(\tria\in\bbT\), we have from
  Lemma~\ref{l:exp-conv} that
  \begin{align*}
    \delta(\mathcal{T})=\eta_{\mathcal{T}}(\mathcal{T}),\qquad\text{i.e.,}\quad
    C_1=1.
  \end{align*}
  \setcounter{enumi}{3}
\item \textbf{Orthogonality.} Indeed, from~\eqref{df:delta}, we
  have for \(\tria,\tria_\star,\tria_{\circ}\in\bbT\) with \(\tria\le\tria_\star\le\tria_{\circ}\) that 
  \begin{align*}
    \delta(\mathcal{T}_{\circ},\mathcal{T}_\star)^2+\delta(\mathcal{T}_\star,\mathcal{T})^2
    &=\eta_{\tria_\star}^2(\tria_\star) -
      \eta_{\tria_{\circ}}^2(\tria_{\circ}) + \eta_{\tria}^2(\tria) -
      \eta_{\tria_{\star}}^2(\tria_{\star})
    \\
    &=\eta_{\tria}^2(\tria)-\eta_{\tria_{\circ}}^2(\tria_{\circ})
    =\delta(\mathcal{T}_\circ,\mathcal{T})^2.
  \end{align*}
\end{enumerate}
Concluding, we have that Theorem~\ref{thm:optimality} and
Theorem~\ref{thm:optRob} apply with the thresholds
\begin{align}\label{eq:thresholds}
  \theta_\star=\frac1{K+1} \qquad\text{and}\qquad \tilde
  \theta_\star=\frac1{K},
\end{align}
respectively.

\begin{remark}\label{rem:context}
  We have verified that our indicators~$\eta$ and the distance
  function~$\delta$ satisfies all stated axioms of
  adaptivity. Nevertheless, we suspect that our
  example~\eqref{df:example} can be realised within the context of
  finite elements for differential equations, as is suggested by
  the following example.
  
  Let \(
  a:\Omega=(0,M)\to \setR_{>}\) piecewise constant with respect to
  \(\mathcal{T}_0\). We consider the following one dimensional
  problem: For \(f\in
  H^{-1}(\Omega)\), find \(u\in H_0^1(\Omega)\) such that
  \begin{align*}
    \forall v\in H_0^1(\Omega)\qquad \int_0^M a\,u'v'\dx =\langle f, v\rangle.
  \end{align*}
  For \(\mathcal{T}\in \bbT\), we chose \(\mathbb{V}(\mathcal{T})\defas \{v\in H_0^1(\Omega)\colon
  v|_{T}\in\mathbb{P}_k, T\in \mathcal{T}\}\) and define
  \(u\in{\mathcal{T}}\in\mathbb{V}(\mathcal{T})\) to be the Galerkin
  approximation of \(u\) in \(\mathbb{V}(\mathcal{T})\). Recalling
  \(H_0^1(\Omega)\hookrightarrow C_0(\bar\Omega)\) since \(d=1\), we
  have that the Lagrange interpolant is stable. Using this,
  standard a posteriori techniques readily show that 
  \begin{align*}
    \int_T
    a(u'-u_{\mathcal{T}}')^2\dx=\frac1{a_{|T}}\|f+(a
    u_{\mathcal{T}}'')\|^2_{H^{-1}(T)}
    \quad\forall
    T\in \mathcal{T}.
  \end{align*}
  An error indicator is then typically obtained by estimating the
  local residuals on the right hand side in a computable way. However,
  their relation to the error is purely local and therefore a
  dependence of the local indicators on the generation
  \(g(T_0(\mathcal{T}))\) as in~\eqref{df:example} is not possible. 
\end{remark}

\subsection{D\"orfler marking}
\label{sec:qo-dorfler}

We recall that $K>1$ is just our reliability constant, i.e. $C_3
=K$, which is related to the threshold by
\begin{align*}
  \tilde{\theta}_\star &= \frac{1}{C_3} = \frac 1K.
\end{align*}


\begin{theorem}\label{thm:example}
  Let $\theta \in (0,1)$ (the D\"orfler parameter) and $s_0>0$ (the
  rate) be given. Then there exist \(\alpha,\beta>0\), \(M\in\setN\),
  and 
  \begin{align*}
    \frac1K=\tilde{\theta}_\star<\theta\quad\text{arbitrary close},
  \end{align*}
  such that Algorithm~\ref{alg:AFEM} with
  optimal D\"orfler marking fails to converge with rate~$s_0$, i.e.
  \begin{align*}
    \sup_{k\in\setN}
    (\#\mathcal{T}_k-\#\mathcal{T}_0)^{s_0}\eta_{\mathcal{T}_k}(\mathcal{T}_k)
    = 
    \sup_{k\in\setN}
    (\#\mathcal{T}_k-\#\mathcal{T}_0)^{s_0}\delta(\mathcal{T}_k)
    = \infty.
  \end{align*}
\end{theorem}
\begin{proof}
  For an arbitrary fixed 
  $\epsilon>0$ we will determine parameters $K,\alpha,\beta$ and $M$
  such that \(\tilde{\theta}_\star=\frac1K \) satisfies
  \begin{align*}
    \tilde{\theta}_\star< \theta < \tilde{\theta}_\star + \epsilon,
  \end{align*}
  i.e.  for some \(\gamma\in(0,\epsilon)\)
  \begin{align}\label{eq:theta-eps}
    \theta=\tilde{\theta}_\star+\gamma=\frac1K+\gamma\qquad\text{or equivalently}\qquad
    K= \frac1{\theta-\gamma}.
  \end{align}
  The constants \(\alpha,\beta>0\) are related to the 
  rate \(s_0\). We fix \(\beta={s_0}>0\) and 
  determined \(\alpha\) at the end of the proof.
  
  In order to introduce the general idea of the proof, we define
  \begin{align*}
    I_{k} := \big[\big((k-1) \text{ mod } M\big)+1, \big((k-1)\text{ mod }
    M \big)+2\big].
  \end{align*}
  Therefore, for any \(j\in\setN\), we have that $I_{0+j}, I_{1+j}, \dots, I_{M-1+j}$ represent the intervals
  $[1,2],\dots, [M,M+1]$ with order shifted by~$j$.
  Below, we will adjust the parameters such that in each iteration
  \(k=0,1,2,\ldots\), the set 
  \begin{align}\label{eq:Mk}
    \mathcal{M}_k &= \set{T_0(\mathcal{T}_k)} \cup \set{T \in
                    \mathcal{T}_k \,:\, T \in I_k}
  \end{align}
  satisfies \emph{optimal D\"orfler
  marking}. In fact, we will have \(\mathcal{M}_k\subset\mathcal{T}_k\) such that
  \(\#\mathcal{M}_k\) is minimal with the property
  \begin{align}
    \label{eq:need-opt-Doerfler}
    \eta^2_{\mathcal{T}_k}(\mathcal{M}_k)
    &= \theta
      \eta^2_{\mathcal{T}_k}(\mathcal{T}_k).
  \end{align}

  Consider first $k=0$.
  It follows from~\eqref{eq:T_0=1/M} that
  \begin{align*}
    \eta^2_{\mathcal{T}_0}(T_0(\mathcal{T}_0))
    &= \frac 1K
      \eta^2_{\mathcal{T}_0}(\mathcal{T}_0),
    \\
    \eta^2_{\mathcal{T}_0}(\mathcal{T}_0|_{[1,M+1]})
    &= \bigg(1-\frac 1K\bigg)
      \eta^2_{\mathcal{T}_0}(\mathcal{T}_0).
  \end{align*}
  Moreover, we have from the definition of our
  indicators~\eqref{df:example} for \(M\in\setN\), that
  \begin{align}\label{eq:eta-decay}
    \eta^2_{\mathcal{T}_0}(\mathcal{T}_0|_{I_j})
    &= 2^{-\beta
      \frac{j-1}M}
      \eta^2_{\mathcal{T}_0}(\mathcal{T}_0|_{[1,2]}), \qquad\text{for
      $j=1,\dots, M$.} 
  \end{align}
  Consequently, it follows from~\eqref{eq:T_0=1/M} that 
  \begin{align*}
    \eta^2_{\mathcal{T}_0}(\mathcal{T}_0)&=K
    \eta^2_{\mathcal{T}_0}(T_0(\mathcal{T}_0))=\frac{K}{K-1}\eta^2_{\mathcal{T}_0}(\mathcal{T}_0|_{[1,M+1]})
    \\
    &=\frac{K}{K-1}\sum_{j=1}^M 2^{-\beta \frac{j-1}{M}}
      \eta^2_{\mathcal{T}_0}(\mathcal{T}_0|_{[1,2]})=\frac{K}{K-1}
      S(\beta,M) \eta^2_{\mathcal{T}_0}(\mathcal{T}_0|_{[1,2]}),
  \end{align*}
  where 
  \begin{align*}
    S(\beta,M) &:= \sum_{j=1}^M 2^{-\beta \frac{j-1}{M}} =
                \frac{1-2^{-\beta}}{1-2^{-\frac \beta M}}.
  \end{align*}
  In other words
  \begin{align}\label{etaR=Seta12}
    \eta^2_{\mathcal{T}_0}(\mathcal{T}_0|_{[1,2]})
    &=
      \frac{\eta^2_{\mathcal{T}_0}(\mathcal{T}_0|_{[1,M+1]})}{S(\beta,M)}
      = \frac{1}{S(\beta,M)} \bigg(1-\frac 1K\bigg)
      \eta^2_{\mathcal{T}_0}(\mathcal{T}_0),
  \end{align}
  and thus the D\"orfler marking
  condition~\eqref{eq:need-opt-Doerfler} reduces to finding
  \(K>1\) and \(M\in\setN\) with
  \begin{align*}
    \frac 1 K + \frac{1}{S(\beta,M)} \bigg( 1 - \frac 1K \bigg) &= \theta
  \end{align*}
  or equivalently (recall~\eqref{eq:theta-eps})
  \begin{align}
    \label{eq:doerfler-equiv1}
    S(\beta,M) &= \frac{1 - \frac 1K}{\theta - \frac 1K}=\frac{1-\theta + \gamma}{\gamma} = \frac1\gamma(1-\theta)+1.
  \end{align}
  Since \(\beta>0\), we have \(S(\beta,M)=1\) and \(\lim_{M \to \infty} S(\beta,M) =
  \infty\) and thus there exist \(M\in\setN\) and
  \(\gamma\in (0,\epsilon)\) satisfying~\eqref{eq:doerfler-equiv1}.

  Overall, 
  thanks to~\eqref{eq:eta-decay} and the fact that
  \(\#\mathcal{T}_0|_{I_0}=\#\mathcal{T}_0|_{I_1}=\cdots=\#\mathcal{T}_0|_{I_{M-1}}=1\), 
  for \(\beta=s_0>0\), we have fixed the parameters $M\in\setN$ and
  \(K>1\), such that~\eqref{eq:theta-eps}
  and~\eqref{eq:doerfler-equiv1} hold. This implies in
  particular optimal D\"orfler marking~\eqref{eq:doerfler}
  for~$k=0$.

  We shall now deal with the case \(k>0\) and let \(k=\ell M+m\in\setN\) with \(\ell\in\setN_0\) and
  \(m\in\{0,\ldots,M-1\}\). It is easy to see from~\eqref{eq:Mk} that
  \begin{align}\label{eq:gen}
    g(T)=
    \begin{cases}
      \ell+1,&\text{if}~T\subset I_{j}~\text{for some}~j\in\{0,\ldots,m-1\}
      \\
      \ell,&\text{if}~T\subset I_{j}~\text{for some}~j\in\{m,\ldots,M-1\}.
    \end{cases}
  \end{align}
  Consequently, we have
  \begin{align*}
    \eta^2_{\mathcal{T}_k}(\mathcal{T}_k|_{I_{m+j}})
    &= 2^{-\beta
      \frac{j-1}M}
      \eta^2_{\mathcal{T}_k}(\mathcal{T}_k|_{I_m}), \qquad\text{for
      $j=1,\dots, M$.} 
  \end{align*}
  Therefore, the relative sizes of the indicators on the intervals
  \(I_{m+j}\) correspond to a cyclic permutation of the initial situation in
  \eqref{eq:eta-decay}. In other words we
  have~\eqref{eq:need-opt-Doerfler}.
  Note that \(I_m=I_k\) by construction and thus
  \begin{align*}
  \eta^2_{\mathcal{T}_k}(\mathcal{T}_k|_{I_k})\ge
  \eta^2_{\mathcal{T}_k}(\mathcal{T}_k|_{I_j}),\qquad j\in\{0,\ldots,M-1\}.
  \end{align*}
  Moreover, it follows from~\eqref{eq:gen} that
  \begin{align*}
    \#\mathcal{T}_k|_{I_k}\le \#\mathcal{T}_k|_{I_j},\qquad j\in\{0,\ldots,M-1\}
  \end{align*}
  and thus the D\"orfler marking is again minimal.

  We turn now to investigate the rate of the algorithm so to fix
  \(\alpha\).  After each $M$ iterations in Algorithm~\ref{alg:AFEM} 
  each element of $[1,M]$ is refined once. Thus, for all
  $\ell \in \setN$
  \begin{align*}
    \#\mathcal{T}_{\ell M} - \# \mathcal{T}_0 \geq 2^\ell
    M.    
  \end{align*}
  Moreover, after $M$ algorithm cycles the element containing zero is
  $M$ times refined and all elements in~$[1,M]$ are refined
  once. Thus, the error indicator of the whole partition decreases
  after $M$ cycles by $2^{-\alpha M -\beta}$, i.e.
  \begin{align*}
    \eta^2_{\mathcal{T}_{\ell M}}(\mathcal{T}_{\ell M})
    &=
      2^{(-\alpha M-\beta)
      \ell}
      \eta^2_{\mathcal{T}_0}(\mathcal{T}_0). 
  \end{align*}
  Therefore, we have with \(\beta=s_0\) that 
  \begin{align*}
    (\# \mathcal{T}_{\ell M} - \# \mathcal{T}_0)^{s_0}
    \eta_{\mathcal{T}_{\ell M}}(\mathcal{T}_{\ell M})
    &\geq M^{s_0} 2^{(s_0  - \frac{\alpha}{2} M- \frac{\beta}{2})\ell}
     =M^{s_0} 2^{(\frac{s_0}2  - \frac{\alpha}{2} M)\ell}.
  \end{align*}
  Choosing  $\alpha\in(0,\frac{s_0}M)$, we have  
  $\frac{s_0}2  - \frac{\alpha}{2} M>0$ and thus
  \begin{align*}
    \sup_{k\in\setN}
    (\#\mathcal{T}_k-\#\mathcal{T}_0)^{s_0}\eta_{\mathcal{T}_k}(\mathcal{T}_k)
    = \infty.
  \end{align*}
  This finishes the proof.
\end{proof}

\subsection{Maximums Strategy}
\label{sec:maximums-strategy}

Another popular refinement strategy is the \emph{maximum
  strategy}. For this the D\"orfler marking~\eqref{eq:doerfler} in
Algorithm~\ref{alg:AFEM} is replaced by 
\begin{align}
  \label{eq:max}
  \mathcal{M}_k:=\big\{T\in\mathcal{T}_k\colon
  \eta_{\mathcal{T}_k}^2(T)\ge \mu
  \max\{\eta_{\mathcal{T}_k}^2(T')\colon T'\in\mathcal{T}_k\}\big\},
\end{align}
for some marking parameter $\mu \in (0,1]$. The strategy requires to
determine the maximal local indicator. Then all elements with
indicators that are up to the factor \(\mu\) maximal are refined.
Obviously, the strategy is getting more selective as closer \(\mu\)
is to one.

The maximum strategy has been analyzed
in~\cite{DieningKreuzerStevenson:2016} and it has been shown that for
any~$\mu \in (0,1]$ the algorithm is \emph{instance
  optimal}.
The term \emph{instance optimality} means that the algorithm
produces meshes with up to a fixed
constant optimal cardinality relativ to the achieved energy
error.  Different from
the D\"orfler marking strategy there is no restriction on
the marking parameter~$\mu$, i.e., all
$\mu \in (0,1]$ are admissible for \emph{instance optimality}.

Let us briefly analyze how the maximum strategy will perform for the
setup of Subection~\ref{sec:setup}. 
It may actually happen in the first iterations
  that elements in \([1,M+1]\) are refined. However, these elements
  are getting then smaller relative to 
  \(\eta^2_{\mathcal{T}_k}(T_0(\mathcal{T}_k))\) due to bisection, thanks to the fact
  that \(|T|=2^{-g(T)}\). Therefore, eventually all
  elements in \([1,M+1]\) are smaller than $\mu\,
  \eta^2_{\mathcal{T}_k}(T_0(\mathcal{T}_k))$. From that point on
only~$T_0(\mathcal{T}_k)$ will be refined and we obtain 
exponential convergence similar as in Lemma~\ref{l:exp-conv}.

This confirms the expected perfomance of the maximum strategy.

\providecommand{\bysame}{\leavevmode\hbox to3em{\hrulefill}\thinspace}
\providecommand{\MR}{\relax\ifhmode\unskip\space\fi MR }
\providecommand{\MRhref}[2]{%
  \href{http://www.ams.org/mathscinet-getitem?mr=#1}{#2}
}
\providecommand{\href}[2]{#2}


\begin{thebibliography}{FKMP13}

\bibitem[BD04]{BiDeV:04}
P.~Binev and R.~DeVore, \emph{Fast computation in adaptive tree
  approximation}, Numer. Math. \textbf{97} (2004), no.~2, 193--217.

\bibitem[BDD04]{BiDaDeV:04}
P.~Binev, W.~Dahmen, and R.~DeVore, \emph{Adaptive finite element
  methods with convergence rates}, Numer. Math \textbf{97} (2004), 219--268.

\bibitem[BDK12]{BelenkiDieningKreuzer:2012}
L.~Belenki, L.~Diening, and C.~Kreuzer, \emph{Optimality of an
  adaptive finite element method for the {$p$}-{L}aplacian equation}, IMA J.
  Numer. Anal. \textbf{32} (2012), no.~2, 484--510. 

\bibitem[BN10]{BonitoNochetto:2010}
A.~Bonito and R.~H. Nochetto, \emph{Quasi-optimal convergence rate of
  an adaptive discontinuous {G}alerkin method}, SIAM J. Numer. Anal.
  \textbf{48} (2010), no.~2, 734--771. 

\bibitem[CFPP14]{CarFeiPagPra14}
C.~Carstensen, M.~Feischl, M.~Page, and D.~Praetorius, \emph{Axioms of
  adaptivity}, Computers \& Mathematics with Applications \textbf{67} (2014),
  no.~6, 1195--1253.

\bibitem[CKNS08]{CaKrNoSi:08}
J.~M. Casc\'on, C.~Kreuzer, R.~H. Nochetto, and K.~G. Siebert,
  \emph{Quasi-optimal convergence rate for an adaptive finite element method},
  SIAM J. Numer. Anal. \textbf{46} (2008), no.~5, 2524--2550.

\bibitem[CN11]{CasconNochetto:11}
J.~M.~Cascón and R.~H.~Nochetto, \emph{{Quasioptimal cardinality of
  AFEM driven by nonresidual estimators}}, IMA Journal of Numerical Analysis
  \textbf{32} (2011), no.~1, 1--29.

\bibitem[CPR13]{CarstensenPeterseimRabus:2013}
C.~Carstensen, D.~Peterseim, and H.~Rabus, \emph{Optimal adaptive nonconforming
  {FEM} for the {S}tokes problem}, Numer. Math. \textbf{123} (2013), no.~2,
  291--308.

\bibitem[CR17]{CarRab17}
C.~Carstensen and H.~Rabus, \emph{Axioms of adaptivity with separate marking
  for data resolution}, SIAM J. Numer. Anal. \textbf{55} (2017), no.~6,
  2644--2665.

\bibitem[DK08]{DieningKreuzer:2008}
L.~Diening and C.~Kreuzer, \emph{Linear convergence of an adaptive
  finite element method for the {$p$}-{L}aplacian equation}, SIAM J. Numer.
  Anal. \textbf{46} (2008), no.~2, 614--638.

\bibitem[DKS16]{DieningKreuzerStevenson:2016}
L.~Diening, C.~Kreuzer, and R.~Stevenson, \emph{Instance optimality
  of the adaptive maximum strategy}, Found. Comput. Math. \textbf{16} (2016),
  no.~1, 33--68.

\bibitem[FFP14]{FeiFuehPraet:2014}
M.~Feischl, T.~F{\"u}hrer, and D.~Praetorius, \emph{Adaptive {FEM} with optimal
  convergence rates for a certain class of nonsymmetric and possibly nonlinear
  problems}, SIAM J. Numer. Anal. \textbf{52} (2014), no.~2, 601--625.

\bibitem[FKMP13]{FeischlKarkulikMelenkPraetorius:13}
M.~Feischl, M.~Karkulik, J.~M. Melenk, and D.~Praetorius, \emph{Quasi-optimal
  convergence rate for an adaptive boundary element method}, SIAM J. Numer.
  Anal. \textbf{51} (2013), no.~2, 1327--1348. 

\bibitem[Gan13]{Tsogtgerel:2013}
T.~Gantumur, \emph{Adaptive boundary element methods with convergence
  rates}, Numer. Math. \textbf{124} (2013), no.~3, 471--516. 

\bibitem[KS11]{KreuzerSiebert:11}
C.~Kreuzer and K.~G. Siebert, \emph{Decay rates of adaptive finite elements
  with {D}\"orfler {M}arking}, Numer. Math. \textbf{117} (2011), no.~4,
  679--716.

\bibitem[KS16]{KreuzerSchedensack:2016}
C.~Kreuzer and M.~Schedensack, \emph{Instance optimal
  {C}rouzeix-{R}aviart adaptive finite element methods for the {P}oisson and
  {S}tokes problems}, IMA J. Numer. Anal. \textbf{36} (2016), no.~2, 593--617.

\bibitem[NSV09]{NoSiVe:09}
R.~H.~Nochetto, K.~G.~Siebert, and A.~Veeser, \emph{Theory of
  adaptive finite element methods: an introduction}, Multiscale, nonlinear and
  adaptive approximation, Springer, Berlin, 2009, pp.~409--542.

\bibitem[NV12]{NochettoVeeser:12}
R.~H.~Nochetto and A.~Veeser, \emph{Primer of adaptive finite element
  methods}, Multiscale and adaptivity: modeling, numerics and applications,
  Lecture Notes in Math., vol. 2040, Springer, Heidelberg, 2012, pp.~125--225.

\bibitem[Ste07]{Stevenson:07}
R.~Stevenson, \emph{Optimality of a standard adaptive finite element method},
  Found. Comput. Math. \textbf{7} (2007), no.~2, 245--269.

\end{thebibliography}
\end{document}